\documentclass[12pt]{amsart}
\usepackage{amsmath,amscd,amsthm,amsfonts, amssymb,amsxtra}
\usepackage[all]{xy} \SelectTips{cm}{}
\usepackage{enumitem}
\usepackage{caption,subcaption,graphicx}
\usepackage{hyperref}
  \hypersetup{colorlinks=true,citecolor=blue}
\CDat

\SelectTips{cm}{}
\newtheorem{thm}{Theorem}[section]
\newtheorem*{un-no-thm}{Theorem}
\newtheorem{cor}[thm]{Corollary}     
\newtheorem{lem}[thm]{Lemma}         
\newtheorem{prop}[thm]{Proposition}

\newtheorem*{ass}{Assumption}
\newtheorem{bigthm}{Theorem}

\newtheorem{bigcor}[bigthm]{Corollary}

\theoremstyle{definition}
\newtheorem{defn}[thm]{Definition}   

\theoremstyle{definition}

\theoremstyle{definition}

\theoremstyle{remark}
\newtheorem{rem}[thm]{Remark}

\newtheorem*{acks}{Acknowledgements}
\newtheorem*{out}{Outline}

\newtheorem{ex}[thm]{Example}

\begin{document}
\title[On Kirchhoff's theorems with coefficients]{On Kirchhoff's theorems with
 coefficients in a line bundle}

\date{\today}

\author[M.J.~Catanzaro]{Michael J.\ Catanzaro}
\address{Department of Mathematics, Wayne State University, Detroit, MI 48202}
\email{mike@math.wayne.edu}
\author[V.Y.~Chernyak]{Vladimir Y.\ Chernyak}
\address{Department of Chemistry, Wayne State University, Detroit, MI 48202}
\email{chernyak@chem.wayne.edu}
\author[J.R.~Klein]{John R.\ Klein}
\address{Department of Mathematics, Wayne State University, Detroit, MI 48202}
\email{klein@math.wayne.edu}

\begin{abstract} We prove `twisted' versions of Kirchhoff's network theorem and 
Kirchhoff's matrix-tree theorem on connected finite graphs. 
Twisting here refers to chains with coefficients in a flat unitary 
line bundle. 
\end{abstract}

\maketitle
\setlength{\parindent}{15pt}
\setlength{\parskip}{1pt plus 0pt minus 1pt}
\def\bdot{\bold .}
\def\Sp{\bold S\bold p}

\def\vo{\varOmega}
\def\smsh{\wedge}
\def\^{\wedge}
\def\flush{\flushpar}
\def\id{\text{\rm id}}
\def\dbslash{/\!\! /}
\def\codim{\text{\rm codim\,}}
\def\:{\colon}
\def\holim{\text{holim\,}}
\def\hocolim{\text{hocolim\,}}
\def\cal{\mathcal}
\def\Bbb{\mathbb}
\def\bold{\mathbf}
\def\simtwohead{\,\, \hbox{\raise1pt\hbox{$^\sim$} \kern-13pt $\twoheadrightarrow \, $}}
\def\codim{\text{\rm codim\,}}
\def\stableto{\mapstochar \!\!\to}
\let\Sec=\S
\def\Z{\mathbb Z}
\def\Q{\mathbb Q}
\def\R{\mathbb R}
\def\C{\mathbb C}
\newcommand{\lra}{\longrightarrow}

\setcounter{tocdepth}{1}
\tableofcontents

\section{Introduction \label{sec:intro}}

It is well-known that the classical result of Kirchhoff on the flow of electricity through
a finite network admits an elegant formulation using algebraic topology 
\cite{Eckmann},\cite{NS},\cite{Roth}. 
For a finite connected $1$-dimensional CW complex
$\Gamma$, the real cellular chain complex $\partial \: C_1(\Gamma;\R) \to C_0(\Gamma;\R)$
is a homomorphism of finite dimensional real inner product spaces 
with orthonormal basis given by the set of cells. When the branches of the network have unit resistance, Kirchhoff's {\it network
theorem} is reflected in the statement that the restricted homomorphism 
\begin{equation} \label{eqn:boundary}
\partial \: B^1(\Gamma;\R) \to B_0(\Gamma;\R)
\end{equation}
is an isomorphism, where $B_0(\Gamma;\R)$ is the vector subspace of zero-boundaries and $B^1(\Gamma;\R)$ is the orthogonal complement to the space of 1-cycles
$Z_1(\Gamma;\R) \subset C_1(\Gamma;\R)$ (when the network has branches of
varying resistance, one rescales the inner product on $C_1(\Gamma;\R)$ accordingly).
Actually Kirchhoff's network theorem does more in that it provides 
a concrete expression for the inverse to
the isomorphism \eqref{eqn:boundary} in terms of the set of spanning trees
of $\Gamma$.  The expression amounts to an explicit formula
for the orthonormal projection of $C_1(\Gamma;\R)$ onto $Z_1(\Gamma;\R)$ in terms
of the set of spanning trees of $\Gamma$.

A companion result of Kirchhoff, which has gotten more press,
is the {\it matrix-tree theorem}, 
which computes the determinant of the restricted
combinatorial Laplacian
\[
\partial\partial^*\: B_0(\Gamma;\R) \to B_0(\Gamma;\R)\, .
\]
Here, $\partial^\ast$
denotes the formal adjoint to the boundary operator (for a slightly different formulation, see \cite[p.~57]{Bollobas}).
In the unit resistance case, the result
says that that $\det(\partial\partial^*)$ equals the product of the number
of vertices with the number of spanning trees of $\Gamma$.
In \cite{CCK}, using ideas from statistical mechanics, we
showed how the matrix-tree theorem can be derived from the general case of the 
network theorem. We also generalized both of these results to higher dimensional
CW complexes.

The main purpose of the current paper is to derive a twisted version of Kirchhoff's theorems. Here the twisting is given by taking  coefficients in a {\it complex line bundle}. 
In physical terms, the twisted version of the network theorem
turns out to model
the flow of current through an electrical network in the presence of {\it fluctuations}. These
fluctuations allow one to compute not only the distribution function of currents, but
also the generating function, via the Fourier transform.

\begin{rem} The physics papers \cite{CKS1} and \cite{CKS2} study the distribution of currents (i.e., homology classes in degree one) on graphs using a non-equilibrium statistical mechanics formalism. The main invariant appearing in these papers is given by averaging currents over stochastic trajectories in a certain long time and low temperature limit. From the physics point of view, one is interested in computing the distribution function. However, it is more convenient to compute the generating function, associated with the probability distribution, which are related via a Fourier transform. The latter can be done by twisting 
the graph Laplacian by a line bundle. It is in this sense that the study of fluctuations corresponds to twisting the Laplacian by a line bundle.
\end{rem}

Our main result is a twisted version of Kirchhoff's projection
formula (Theorem \ref{bigthm:twisted-kirchhoff}). 
As an application, we will deduce a twisted version of the matrix-tree theorem (Theorem \ref{bigthm:twmtt}). 
Suitably reformulated, our twisted matrix-tree theorem is actually
a result of Forman \cite[eq.~(1)]{Forman} which
we first learned about in a recent paper of Kenyon \cite[thm.~5]{Kenyon}. 
Forman's proof is combinatorial, using an explicit expression for the determinant
in terms of symmetric groups.  
Kenyon's proof relies on the Cauchy-Binet theorem. By contrast, 
our approach is inspired by statistical mechanical ideas and closely follows
the untwisted version appearing in  \cite{CCK}.

\subsection*{Graphs}
A {\it graph} $\Gamma$ is a CW complex of dimension one. 
We let $\Gamma_0$ denote the set of $0$-cells and $\Gamma_1$ the 
set of $1$-cells.
A $0$-cell is called a {\it vertex} and a $1$-cell is called an {\it edge.}
The entire structure of $\Gamma$ is given by a function
\begin{equation}
\label{eqn:d-zero_d-one}
(d_0,d_1)\: \Gamma_1 \to \Gamma_0 \times \Gamma_0\, ,
\end{equation}
which sends an edge $b$ to its initial and terminal endpoints (where the edge
$b$ is oriented using its characteristic map $\chi_b\: [0,1] \to \Gamma$).
Given the function \eqref{eqn:d-zero_d-one}, one can reconstruct $\Gamma$ 
by taking
\[
\Gamma_0 \cup (\Gamma_1 \times [0,1])
\]
where the union is amalgamated over the map $\Gamma_1 \times \{0,1\}\to \Gamma_0$
given by $(b,0) \mapsto d_0 (b)$ and $(b,1) \mapsto d_1(b)$.

A {\it loop edge} is an edge such that $d_0(b) = d_1(b)$.
If $b$ is not a loop edge, it is said to be {\it regular}.

\subsection*{Flat line bundles on graphs} 
A {\it flat complex vector bundle} $\rho$ on a graph $\Gamma$ is a rule which assigns to each vertex
$i \in \Gamma_0$ a finite rank complex vector space $V_i$ over $\Bbb C$ and to each edge 
$b$ with $(d_0 b,d_1b) = (i,j)$ an isomorphism 
\[
\rho_b\: V_i \to V_j\, .
\] 
We say that $\rho$ is {\it unitary} if each $V_i$ is a hermitian inner product space
and each $\rho_b$ is unitary. In this paper we will deal exclusively with 
the rank one case, i.e., flat complex line bundles. Henceforth, 
we simplify terminology and refer to $\rho$ as a {\it line bundle.}\footnote{Strictly speaking, what we
have defined here is really the notion of a transport operator on $\Gamma$
associated with a flat connection, but we will not need to worry about this distinction.}

Given a line bundle $\rho$, by choosing a non-zero vector $u_i$ in $V_i$ having  unit norm, we can identify $V_i$ with $\Bbb Cu_i$, the complex vector space
spanned by $u_i$. Consequently, there is no loss in generality in assuming that 
$V_i = \Bbb C$ for every $i\in \Gamma_0$. In this instance $\rho_b\: \Bbb C \to \Bbb C$ is
given by multiplication by a unit complex number, which by abuse of notation we
denote as $\rho_b$. With respect to these choices, 
$\rho$ is given by a function $\Gamma_1 \to U(1)$.

Recall that a circuit $C$ of $\Gamma$ is a simple closed path.
An orientation of $C$ consists of a choice of direction for traversing $C$.

\begin{defn} If $C$ is an oriented circuit of $\Gamma$, then 
the {\it holonomy} of $\rho$ along $C$ is given by the product 
\[
\rho_C := \prod_{b \in C_1} \rho^{s_b}_b
\]
where $s_b = \pm 1$  according as to whether the orientation of $C$ is
the same as the orientation of $b$. 
\end{defn}

If $\bar C$ denotes $C$ with its reverse orientation, then
$\rho_{\bar C} = {\rho}^\ast_C$, where ${\rho}^\ast_C$ denotes
the complex conjugate of $\rho_{C}$.

More generally, suppose
$A \subset \Gamma$ is a subgraph with the following property
that each component $A_\alpha$ of $A$ has trivial Euler characteristic.
Then $A_\alpha$ has a unique circuit $C_\alpha$. Then $A$ has a preferred
set of circuits.
Assigning to $C_\alpha$ of $A$ is an arbitrary orientation,  we set
\[
\hat\rho_A := \prod_{\alpha} (\rho_{C_\alpha}-1)(\rho^\ast_{C_\alpha}-1) = \prod_{\alpha} (2-\rho_{C_\alpha}-\rho^\ast_{C_\alpha})\, ,
\]
where $\alpha$ ranges over the components of $A$. This last expression is 
well-defined and independent
of the choice of orientation for the circuits.  It  is also a real number.

\subsection*{The twisted chain complex}  
For $i = 0,1$, let $C_i(\Gamma;\rho)$ denote the $\Bbb C$-vector space having basis
$\Gamma_i$. Define the twisted boundary operator
\[
\partial\: C_1(\Gamma;\rho) \to C_0(\Gamma;\rho) 
\]
by mapping an edge $b$ to the 
vector $\rho_b d_0(b) - d_1(b)$ and extending linearly. The homology of this two-stage complex
is denoted by $H_\ast(\Gamma;\rho)$. It is invariant with respect to barycentric subdivision.
That is, if $\Gamma'$ is the barycentric subdivision of $\Gamma$ and $\rho'$ is a line
bundle on $\Gamma'$ such that $\rho_b = \rho'_{b_0}\rho'_{b_1}$ when $b = b_0b_1$ is the
subdivision of an edge $b$, then $H_\ast(\Gamma;\rho) \cong H_\ast(\Gamma';\rho')$.
Note that $H_1(\Gamma;\rho)$ is a subspace of $C_1(\Gamma;\rho)$ consisting of
the cycles. 

\begin{rem} This is the traditional notation. It is imprecise
since each vector space $C_i(\Gamma;\rho)$ does not depend
on $\rho$ whereas the boundary operator $\partial$ does. A more precise notation would
write the complex as $\partial_\rho\: C_1(\Gamma;\C) \to C_0(\Gamma;\C)$. 
\end{rem}

If $A\subset \Gamma$ is a subcomplex, we have the relative
chain complex $C_\ast(\Gamma,A;\rho)$ which is the quotient complex 
$C_\ast(\Gamma;\rho)/C_\ast(A;\rho)$. It has a basis consisting of the cells
of $\Gamma$ which are not in $A$. 

\subsection*{The resistance operator} 
A {\it resistance function} is a map $r\: \Gamma_1 \to \R_+$ which 
assigns to an edge $b$ a resistance $r_b > 0$. Associated with $r$ is the
{\it resistance operator}
\[
R\: C_1(\Gamma;\rho) \to C_1(\Gamma;\rho)\, ,
\]
which on basis elements is defined by  $b \mapsto r_b b$. 

\subsection*{The standard and modified inner products}
The {\it standard Hermitian inner product} 
on $C_1(\Gamma;\rho)$, denoted $\langle \phantom{a},\phantom{a} \rangle$,
is given on basis elements $b,b' \in \Gamma_1$ by 
\[
\langle b, b'\rangle := \delta_{bb'}\, ,
\]
where $\delta_{bb'}$ is Kronecker delta. 

Associated with the resistance operator $R$ is the {\it 
modified inner  product} on $C_1(\Gamma;\rho)$, denoted 
$\langle \phantom{a},\phantom{a} \rangle_R$, 
 is given by
\[
\langle b, b'\rangle_R := r_b \langle b,b'\rangle = \delta_{bb'}r_b\, .
\]

\subsection*{Twisted spanning trees}
From now on we assume that $\Gamma$ is connected
and finite.

\begin{defn}
A {\it $\rho$-spanning tree} for $\Gamma$ is a subcomplex $T \subset \Gamma$ such that
\begin{itemize}
\item $T_0 = \Gamma_0$,
\item $H_1(T;\rho) = 0$, and
\item The homomorphism 
$H_0(T;\rho) \to H_0(\Gamma;\rho)$ induced by the  inclusion
is an isomorphism.
\end{itemize}
\end{defn}

\begin{rem} When $\rho$ is the trivial line bundle, we recover the usual
notion of spanning tree. In the next section we characterize 
the $\rho$-spanning trees of $\Gamma$.
\end{rem}

We henceforth make the following assumption:

\begin{ass} The vector space $H_0(\Gamma;\rho)$  is trivial. 
\end{ass}

\begin{rem}
The triviality of $H_0(\Gamma;\rho)$ is equivalent to the statement that the holonomy over each cycle 
of $\Gamma$ is non-trivial. In the case of the twisted matrix-tree theorem
(Theorem \ref{bigthm:twmtt} below),
this assumption doesn't  cause additional restrictions on generality: 
if $H_0(\Gamma;\rho)$
is non-trivial, then the twisted Laplacian has trivial determinant.
\end{rem}

\begin{defn} \label{weights} The {\it weight} of a $\rho$-spanning tree $T$ is the
real number
\[
w_T := \hat \rho_T \prod_{b \in T_1} r_b^{-1}\, .
\]
\end{defn}

\begin{rem}  If we delete the factor $\hat \rho_T$ from
the above expression, we obtain the  weights appearing in 
classical untwisted version of the Kirchhoff formula (cf.\ \cite{NS}). 
\end{rem}

\subsection*{The operator $\bar T$}
Given a $\rho$-spanning tree $T$, we define an operator
\[
\bar T \: C_1(\Gamma;\rho) \to H_1(\Gamma;\rho)
\]
as follows: if $b\in T_1$ then $\bar T(b) = 0$. If $b \in \Gamma_1 \setminus T_1$,
 we form the graph $T \cup b$. Then $\dim_{\C} H_1(T\cup b;\rho)= 1$. Let
 $c\in H_1(T\cup b;\rho)$ be a non-zero vector, and set $t_b = \langle c,b\rangle$.
We set $\bar T(b) := c/t_b$. This does not depend on the choice of $c$. Note that
$H_1(T\cup b;\rho) \to H_1(\Gamma;\rho)$ is an inclusion, so this definition makes sense.

\begin{rem} It will be useful to have an alternative description of $\bar T$. 
Assume $b\in \Gamma_1 \setminus T_1$. The homology class
$[\partial b] \in H_0(T;\rho) = 0$ is trivial, so $\partial b \in C_0(T;\rho)$ bounds
a chain $u \in C_1(T;\rho)$. Then $c := b - u \in C_1(T\cup b;\rho)$ is a cycle such that
$t_b = \langle c,b\rangle = 1$. In this case $\bar T(b) = c$.
\end{rem}

\subsection*{The main results}
The twisted version of Kirchhoff's network
theorem will be a consequence of having a concrete description of the projection
operator from twisted 1-chains to twisted 1-cycles.
 
\begin{bigthm}[Twisted Projection Formula] 
\label{bigthm:twisted-kirchhoff} With respect to the modified
inner product $\langle \phantom{a},\phantom{a}\rangle_R$, 
the hermitian projection of $C_1(\Gamma;\rho)$ onto the subspace $H_1(\Gamma;\rho)$
is given by 
\[
  {\displaystyle
    \frac{1}{\Delta}\sum_T w_T \bar T \, ,
  }
\]
where $T$ ranges over the $\rho$-spanning trees of $\Gamma$ and
$\Delta = \sum_T w_T$.
\end{bigthm}

Our  twisted version of Kirchhoff's network theorem is

\begin{bigcor}[Twisted Network Theorem] \label{cor:twisted_Kirchhoff_formula}
Given a vector $\mathbf V \in C_1(\Gamma;\rho)$, there is only one vector
$z\in Z_d(\Gamma;\rho)$ such that $\mathbf V - Rz \in B^{1}(\Gamma;\rho)$. 
Furthermore, for each edge
$b$, we have
\[
  \displaystyle{
    \langle z,b\rangle = \frac{1}{\Delta} \sum_T \frac{w_T}{r_b} \langle \mathbf V,\bar T(b)\rangle \, . }
\]
\end{bigcor}

\begin{rem}  In the untwisted case ($\rho = 1$), this is the formulation
of Kirchhoff's network theorem that is found in \cite{NS}. The expression  
$\langle\mathbf V,b\rangle$ is called the voltage source on the branch
$b$ and $\langle z,b\rangle$ is the current residing on $b$.
\end{rem}

Let 
\[
\partial^*_R\: C_0(\Gamma;\rho) \to C_1(\Gamma;\rho)
\]
be the formal adjoint to the boundary operator $\partial\: C_1(\Gamma;\rho) 
\to C_0(\Gamma;\rho)$ with respect to the standard hermitian inner product on 
$C_0(\Gamma;\rho)$ and the modified one on $C_1(\Gamma;\rho)$ as determined
by the resistance operator $R$.

The following is the result of Forman \cite[eq.~(1)]{Forman} that was alluded to above.

\begin{bigthm}[Twisted Weighted Matrix-Tree Theorem] \label{bigthm:twmtt}  
\[
\det (\partial \partial^*_R\:C_0(\Gamma;\rho) \to C_0(\Gamma;\rho))  = \sum_{T} w_T \, ,
\] 
where $T$ ranges over all $\rho$-spanning trees, and $w_T$ is as in Definition \ref{weights}.
\end{bigthm}

The case $R = 1$ is worth singling out. We use 
the notation $\partial^* = \partial^*_R$ in this case.

\begin{bigcor}[Twisted Matrix-Tree Theorem] 
\[
\det (\partial \partial^*)  = \sum_{T} \hat \rho_T \, ,
\]
where $T$ ranges over all $\rho$-spanning trees.
\end{bigcor}

\begin{rem} A natural question is whether versions of Theorems 
\ref{bigthm:twisted-kirchhoff} and \ref{bigthm:twmtt} exist
for higher rank bundles on $\Gamma$.  We don't think this
is likely, 
since our approach relies heavily
on the fact that $U(1)$ is abelian.
\end{rem}

\begin{out} In \S~\ref{sec:spanning} we develop foundational
material on $\rho$-spanning trees. \S~\ref{section:network_proof}
contains the proofs of Theorem \ref{bigthm:twisted-kirchhoff} and
Corollary \ref{cor:twisted_Kirchhoff_formula}. In \S~\ref{sec:matrix-tree}, we prove Theorem \ref{bigthm:twmtt} using Theorem \ref{bigthm:twisted-kirchhoff} and the low temperature limit
argument of \cite{CCK}.
\end{out}

\begin{acks} The authors wish to Nikolai Sinitsyn for discussions related to
the physical interpretation of the twisted network theorem. We are also indebted
Misha Chertkov for hosting the first author and
to Andrei Piryatinski for his unlimited hospitality.
We thank the Los Alamos Center for Nonlinear Studies and the T-4 division for partially supporting this research. This material is based upon work supported by the National
Science Foundation under Grant Nos.\ CHE-1111350 
and DMS-1104355. 
\end{acks}

\section{Properties of twisted spanning trees \label{sec:spanning}}

The following lemma characterizes the structure of $\rho$-spanning trees of $\Gamma$. We remind the reader we have made the assumption throughout that 
$H_0(\Gamma;\rho) = 0$.

\begin{lem} \label{lem:characterizeT}  A subcomplex 
$T \subset \Gamma$
is a $\rho$-spanning tree  if and only if 
\begin{itemize}
\item $T_0 = \Gamma_0$, 
\item each connected component
$T^\alpha$ of $T$ has trivial Euler characteristic, i.e., $T^\alpha$ possesses a unique 
circuit, $C_\alpha$ and
\item the holonomy around $C_\alpha$ is non-trivial.
\end{itemize}
\end{lem}

\begin{rem} \label{rem:CSRF} A {\it cycle-rooted spanning forest (CRSF)} is a subcomplex $T$ of $\Gamma$
satisfying the first two conditions listed in Lemma \ref{lem:characterizeT}, i.e.,
a $\rho$-spanning tree is a CRSF additionally satisfying the condition that
the holonomy around circuits is non-trivial
(cf. Fig.~\ref{fig:CSRF}, \cite[4.1]{Kenyon}).
\end{rem}

\begin{figure}
\includegraphics[scale=.5]{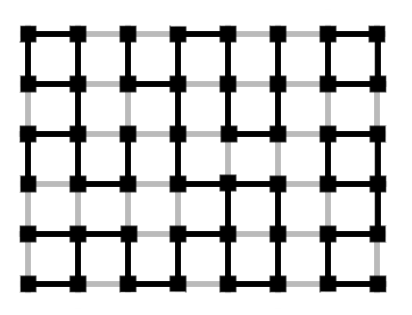}
\captionsetup{name=Fig.\!}
\caption{ \label{fig:CSRF} A lattice graph equipped with CRSF having four components (cf. Remark \ref{rem:CSRF}).
The edges of the graph are the gray lines. The edges of the CRSF
are indicated in black.}
\label{spanning_tree_graphic}
\end{figure}

\begin{proof}[Proof of Lemma \ref{lem:characterizeT}]  Assume $T$ is a $\rho$-spanning tree. 
The assumption $H_0(\Gamma;\rho) = 0$ implies that
$H_0(T;\rho) = 0$ and the latter implies $H_0(T^\alpha;\rho) = 0$ since
$H_0(T;\rho) = \oplus_\alpha H_0(T_\alpha;\rho)$. Similarly, 
$H_1(T;\rho) = 0$ implies $H_1(T^\alpha;\rho) = 0$. Hence the chain complex
\[
\partial\: C_1(T^\alpha;\rho) \to C_0(T^\alpha;\rho)
\]
is acyclic. In particular, the number of edges of  $T^\alpha$ 
equals the number of vertices, so the Euler characteristic of $T^\alpha$ is trivial.
Orient the unique circuit $C_\alpha$ and let the holonomy around
$C_\alpha$ be denoted $\rho_\alpha$. 
Then independence of twisted cohomology with respect to subdivision yields
$H_0(T^\alpha;\rho)  = H_0(C_\alpha;\rho) = H_0(S^1;\rho_\alpha)$, where we
are thinking of $S^1$ as a graph with one vertex and one edge and where the line
bundle is given by $\rho_{\alpha} $. An easy calculation shows $H_0(S^1;\rho_C)$
is the cokernel of the map $(\rho_\alpha - 1)\cdot\: \Bbb C \to \Bbb C$.  Hence
the triviality of $H_\ast(T^\alpha;\rho_\alpha)$ is equivalent to the statement $\rho_\alpha \ne 1$.

Conversely, given $T$ satisfying the three conditions, the second and
third conditions imply  $H_*(T;\rho) = \oplus_\alpha H_\ast(T^\alpha;\rho)$ is trivial.
Hence $T$ is a $\rho$-spanning tree.
\end{proof}

\begin{lem} $\Gamma$ has a $\rho$-spanning tree.
\end{lem}

\begin{proof}  Call an edge $b$ of $\Gamma$ {\it essential} if there is a
cycle $z \in H_1(\Gamma;\rho) \subset C_1(\Gamma;\rho)$ such that $\langle b,z\rangle \ne 0$.
If there is no such edge, then it is straightforward to check that $\Gamma$ is
a $\rho$-spanning tree.

Assume then that there is an essential edge $b$. Let $Y$ be the effect
of removing (the interior of) $b$ from $\Gamma$. Then we have a short exact sequence
\[
0 \to H_1(Y;\rho) \to H_1(\Gamma;\rho) \to H_1(b,\partial b;\rho) \to 
H_0(Y;\rho) \to 0
\]
and the condition $\langle b,z\rangle \ne 0$ implies that the homomorphism
$H_1(\Gamma;\rho) \to H_1(b,\partial b;\rho)$ is non-trivial 
(note that $H_1(b,\partial b;\rho) \cong \C$). It follows that $\dim_{\C} H_1(Y;\rho) < \dim_{\C} H_1(\Gamma;\rho)$
and $H_0(Y;\rho) = 0$. We do not require that $Y$ be connected.
We now replace $\Gamma$ by $Y$ and iterate this construction until we obtain
a subcomplex $T$ having no essential cells and $H_\ast(T;\rho) = 0$. Then $T$ is
a $\rho$-spanning tree.
\end{proof} 

\begin{lem}\label{lem:basis} Fix a $\rho$-spanning tree $T$ and let $b_1,\dots,b_k$ be the
set of edges of $\Gamma_1 \setminus T_1$. Then $\{\bar T(b_1),\ldots,\bar T(b_k)\}$
is a basis for $H_1(\Gamma;\rho)$.
\end{lem}
 
 \begin{proof} The homomorphism $H_1(\Gamma;\rho) \to H_1(\Gamma,T;\rho)$
an isomorphism. Furthermore, $H_1(\Gamma,T;\rho) = C_1(\Gamma;T;\rho)$
 has basis $\{b_1,\ldots,b_k\}$. The inverse homomorphism sends $b_i$ to $\bar T(b_i)$.
 \end{proof}
 
 \begin{cor} \label{for:idempotent} For any $z\in H_1(\Gamma;\rho)$, 
 we have $\bar T(z) = z$.
 \end{cor}
 
 \begin{proof} The definition of $\bar T$ shows
$\bar T^2(b_i) = \bar T(b_i)$. 
 Write $z = \sum_i a_i \bar T(b_i)$. Then 
 \[
 \bar T(z) = \sum_i a_i \bar T^2(b_i) = \sum_i a_i \bar T(b_i) = z\, . \qedhere
 \]
  \end{proof}

Given a $\rho$-spanning tree $T$, consider an edge $b_i \in \Gamma_1 \setminus T_1$ as well
as an edge $b_j \in  T_1$. Let $U = (T \setminus b_j) \cup b_i$. 

\begin{lem} $U$ is a $\rho$-spanning tree if and only if
 $\langle \bar T(b_i),b_j\rangle \ne 0$.
\end{lem}

\begin{proof} Throughout this 
proof we use local coefficients in $\rho$ but suppress this
from the notation. We have an exact sequence
\[
0 \to H_1(T \setminus b_j) \to H_1(T) \to H_1(b_j,\partial b_j) \to H_0(T \setminus b_j) \to 0\, ,
\]
where we are using the fact that $H_0(T) = 0$.
Since $H_1(T) = 0$ and $\dim_{\C} H_1(b_j,\partial b_j)= 1$,
we infer that $H_1(T \setminus b_j) = 0$ and  $\dim_{\C} H_0(T\setminus b_j)= 1$.

The inclusion $U \subset T \cup b_i$ induces another exact sequence
\[
0 \to H_1(U) \to H_1(T\cup b_i) \to H_1(b_j,\partial b_j) \to H_0(U) \to 0\, ,
\]
and the homomorphism $H_1(T\cup b_i) \to H_1(b_j,\partial b_j)$ is a map
of rank one vector spaces that is induced by sending
the preferred cycle $c \in H_1(T\cup b)$ 
to $\langle \bar T(b_i),b_j\rangle$ with respect to the preferred identification
$H_1(b,\partial b) \cong \C$. Consequently, $U$ is a $\rho$-spanning tree if and only if
$\langle \bar T(b_i),b_j\rangle \ne 0$.
\end{proof}

\begin{prop}\label{prop:self-adjoint} With $b_i, b_j, T, U$ as above, we have
\[
\hat \rho_T \langle \bar T(b_i),b_j\rangle = \hat \rho_U \langle b_i,\bar U(b_j)\rangle \, .
\]
\end{prop}

\begin{rem} Proposition \ref{prop:self-adjoint} will be
a key step in verifying the
Twisted Projection Formula (Theorem \ref{bigthm:twisted-kirchhoff}).  
Although we will have managed to reduce most of the 
argument to algebraic topology, 
we cannot completely eliminate combinatorics
from the proof entirely (the same is true with respect to the classical theorem; see \cite{NS}). However, Proposition \ref{prop:self-adjoint} effectively minimizes the role of combinatorics to a kind of general and relatively simple statement.
\end{rem}

\begin{proof}[Proof of Proposition \ref{prop:self-adjoint}] There are two cases to consider:
  either $b_i$ is attached to two distinct components of $T$ or $b_1$ is
  attached to a single component of $T$. We proceed by direct calculation
  in either case. Figure~\ref{fig:prop-cases} gives a visualization of the cases at hand.
  
  {\it Case 1:} Assume that $b_i$ is attached to two distinct components
of $T$, say $A$ and $B$. By switching the roles of $A$ and $B$ if necessary, we
may suppose that $(d_0(b_i),d_1(b_i)) = (w,v)$, where $v$ lies in $A$
and $w$ lies in $B$. Without loss in generality assume that $b_j$ lies in $A$. 
Let $C$ be the unique circuit of $A$ and $C'$ the unique circuit of $B$.

Then
\[
v = \partial c 
\]
for $c \in C_1(A;\rho)$. We may then write $c = c_0 + \alpha b_j$, where $\alpha \in \C$ and
$\langle c_0,b_j\rangle = 0$. Similarly, we write $w = \partial d$, where $d \in C_1(B;\rho)$.
Then $\langle \bar T(b_i),b_j\rangle$ equals
\[
\langle b_i - (\rho_{b_i} d - c_0 - \alpha b_j), b_j \rangle = \alpha \, ,
\]
since $\partial(\rho_{b_i} d - c_0 - \alpha b_j) = \partial b_i$ and
$\rho_{b_i} d - c_0 - \alpha b_j$ is a chain of $T$.

A similar calculation shows $\langle b_i,\bar U (b_j)\rangle$ equals 
\[
\langle b_i, b_j - \tfrac{\rho_{b_i}d-b_i-c_0}{\alpha} \rangle =  (\alpha^{-1})^\ast\, .
\]
In order to compute $\alpha$, it is enough to identify the 1-chain $c \in C_1(A;\rho)$ whose boundary equals $v$, since then $\langle c,b_j \rangle = -\alpha$.

To find $c$ we rename $v = v_1$ and choose a vertex $v_k$ on the unique
circuit of $A$ together with an embedded path of edges $e_1,\dots e_k$ which connects $v_1$ to $v_k$. Without loss in generality, we can assume that none of the edges $e_i$ lies in the unique cycle of $A$. Let $e_{k+1},\dots, e_{n}$ denote the sequence of edges given by
the traversing the unique cycle of $A$ such that $v_k$ is a vertex of both $e_{k+1}$ and $e_n$.
Then $c$ is a linear combination of the edges $e_i$ which can be explicitly computed 
using the fact that $\partial e_i = \rho_{e_i}d_0(e_i) - d_1(e_i)$. 
Then a straightforward calculation yields the expression for the component of $c$ along the edge $e_i$ as
\begin{equation}\label{eqn:component}
\langle c,e_i\rangle  = -\frac{\rho_1^{s_1}\ldots \rho_{i-1}^{s_{i-1}}}{\rho_A - 1}\, ,
\end{equation}
where $\rho_i := \rho_{e_i}$ and $s_i = \pm 1$ according as to whether $e_i$ points
in the direction of the path or not (we have also oriented $A$ in a way
that is  compatible with our choice of path).
 In particular, $b_j = e_\ell$ for some index $\ell$, 
so 
\[
\alpha = 
\frac{\rho_1^{s_1}\ldots \rho_{\ell-1}^{s_{\ell-1}}}{\rho_A - 1} \, ,
\]
Since $\hat \rho_T =\hat\rho_A\hat\rho_B\hat  \rho'$, where  $\hat\rho'$
is the product of the $\hat\rho_{T^\alpha}$ ranging over the 
remaining components of $T$, we have  
\[
\hat\rho_T\alpha = \rho_1^{s_1}\ldots \rho_{i-1}^{s_{\ell-1}}(\rho_A^*-1)\hat\rho_B \hat \rho'\, .
\]
Since $\hat\rho_U = \hat\rho_B\hat \rho'$ we see
\[
 \hat\rho_U(\alpha^{-1})^* = \hat\rho_T \alpha \, ,
 \] 
which concludes Case 1.
\medskip

\begin{figure}
\centering
\begin{subfigure}{.45\textwidth}
  \centering
  \includegraphics[scale = 0.4]{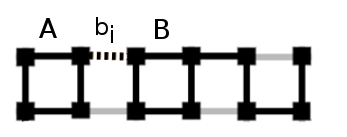}
  \caption{ $b_i$ bridges two distinct \\ components of $T$.}
  \label{fig:sub1}
\end{subfigure}%
 \begin{subfigure}{.5\textwidth}
  \centering
  \includegraphics[scale = 0.5]{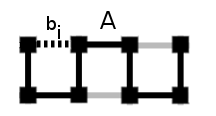}
  \caption{ $b_i$  attached to a single component of $T$.}
  \label{fig:sub2}
\end{subfigure}
\caption{ The two cases of Proposition~\ref{prop:self-adjoint}}
\label{fig:prop-cases}
\end{figure}
\noindent {\it Case 2:} In this instance $b_i$ is attached to a single component $A$ of $T$.
In this case we need to find a 1-chain $c$ of $A$ such that $\partial c = \partial b_i$.
Arguing in an analogous way as in the beginning of Case 1, if we set
\begin{equation} \label{eq:alpha-Case-2}
\langle \bar T(b_i),b_j \rangle = \alpha
\end{equation}
then it follows that
\begin{equation}\label{eq:alpha-inverse-conj-Case-2}
 \langle \bar b_i,\bar U(b_j) \rangle = (\alpha^{-1})^* \, .
\end{equation}

Suppose  $d_1(b_i) = v$ and $d_0(b_i) = w$ (where it is possible that $v = w$). 
We select a simple path $e_1,\ldots,e_k$ of edges of $A$ such that $v$ meets $e_1$
and $w$ meets $e_k$. Let us rename $b_i$ as $e_{k+1}$.
Then $c$ is a linear combination of the edges $e_i$ for $1\le i \le k$. Let $C$ denote the unique circuit of $A$. Then $A$ decomposes as
\[
A_- \cup C \cup A_+
\]
in which $A_-$ is a connected subgraph of $A$ that meets the vertex $v$ and
$A_+$ is the connected subgraph of $A$ which meets the vertex $w$.
For a given index $i$, consider the expressions 
\[
\alpha_- := \prod_{1\le j < i} \rho_j^{s_j} \quad \alpha_+ := \prod_{i \le j \le k+1} \rho_j^{s_j}\, .
\]
A calculation
similar to that appearing in Case 1 gives, for $0\le i \le k$,
\begin{equation} \label{eqn:compute-bar-T}
\langle c,e_i \rangle = 
\begin{cases} 
\alpha_-\, , & \text{ if } e_i \subset A_- \, ,\\
(\rho_A - 1)^{-1}(\alpha_+ - \alpha_-)\, ,
& \text{ if } e_i\subset C \, , \\
\alpha_+\, , & \text{ if } e_i \subset A_+ \, .
\end{cases}
\end{equation} 
Then if $\beta_j = e_\ell$ for some $\ell$, 
we have $\langle \bar T(\beta_i),\beta_j \rangle = -\langle c,e_\ell\rangle$.
As before, we have $\hat\rho_T = \hat \rho_A \hat\rho'$,  where
$\hat \rho'$ is a product of $\hat \rho_{T^\alpha}$ for $T^\alpha$ ranging
over the other components of $T$. Consequently,
\[
\langle \hat\rho_T \bar T(\beta_i),\beta_j \rangle = \hat \rho_A \hat\rho'\alpha
\]
where $\alpha = -\langle c,e_\ell\rangle $ is explicitly
given by Eqn.~\eqref{eqn:compute-bar-T}.

The remainder of the  argument is just as in Case 1. 
A straightforward calculation that we omit shows
\begin{equation} \label{eqn:hat-rho-U}
\hat \rho_U = \begin{cases}
\hat \rho_T & b_j \subset A_- \cup A_+\, , \\
(2+ \alpha_+\alpha_-^* - \alpha_+^*\alpha_-)\hat\rho' & b_j \subset C
\end{cases} 
\end{equation}
Then use Eq.~\eqref{eq:alpha-inverse-conj-Case-2} and Eq.~\eqref{eqn:compute-bar-T}
to identify the product $\hat\rho_U(\alpha^{-1})^*$. We infer
that it coincides with
$\hat\rho_T \alpha$, thereby completing the proof.
\end{proof}

\section{Proof of Theorem \ref{bigthm:twisted-kirchhoff} and 
Corollary \ref{cor:twisted_Kirchhoff_formula} \label{section:network_proof}}

\begin{lem} \label{lem:TU} For distinct edges
$b_i,b_j \in \Gamma_1$, let $\cal T_{ij}$ be the set of $\rho$-spanning
 trees such that $\langle \bar T(b_i),b_j \rangle \ne 0$.  Then
 \[
 \sum_{T\in \cal T_{ij}}w_T \langle \bar T(b_i), b_j\rangle_R
 =  \sum_{U\in \cal T_{ji}}w_U \langle b_i, \bar  U(b_j)\rangle_R \, .
 \]
 \end{lem}

\begin{proof}  From the definition of the weights, have
\begin{equation}\label{eq:r-w-rho-relation}
\frac{r_jw_T}{\hat\rho_T} = \frac{r_i w_U}{\hat \rho_U}\, .
\end{equation} 
Recall that $\langle \bar T(b_i),b_j\rangle_R = r_j
\langle \bar T(b_i),b_j\rangle$. Using Eq.~\eqref{eq:r-w-rho-relation} and
Proposition \ref{prop:self-adjoint}, we infer
\[
w_T \langle\bar T(b_i),b_j\rangle_R = w_U \langle b_i,\bar U(b_j)\rangle_R \, .
\]
Now sum up over all $T\in \cal T_{ij}$.
\end{proof}

\begin{proof}[Proof of Theorem \ref{bigthm:twisted-kirchhoff}]
Consider the operator $F:= \sum_T w_T\bar T$, where the sum is over
all $\rho$-spanning trees of $\Gamma$.
For any pair of edges $b_i$ and $b_j$ of $\Gamma$ we have
\begin{align*}
\langle \sum_T w_T\bar T(b_i),b_j\rangle_R
& = \sum_{T\in \cal T_{ij}}   w_T\langle\bar T(b_i),b_j\rangle_R \\
& =\sum_{U\in \cal T_{ji}}w_U \langle b_i, \bar U(b_j)\rangle_R \qquad
\text{ by Lemma \ref{lem:TU} } , \\
& = \langle b_i ,\sum_U w_U \bar U(b_j)\rangle_R \\
& = \langle b_i ,\sum_T w_T \bar T(b_j)\rangle_R
\end{align*}
Hence $F$ is self-adjoint in the modified inner product.

If $z  \in Z_1(\Gamma;\rho)$, then using Corollary \ref{for:idempotent},
we have 
\[
F(z) = (\sum_T w_T)\bar T(z)  =  (\sum_T w_T)z =: \Delta z
\]
Consequently,
$(1/\Delta) F$ restricts to the identity on $Z_d(X;\rho)$. As $(1/\Delta) F$
is self-adjoint, it is the Hermitian projection in the modified inner product.
\end{proof}

\begin{proof}[Proof of Corollary \ref{cor:twisted_Kirchhoff_formula}]  Let $z$ be the Hermitian projection of $R^{-1}\mathbf V$ in the modified inner product. Then $R^{-1}\mathbf V - z \in B^{d}_R(\Gamma;\rho)$, i.e.,
\[
0 = \langle R^{-1}\mathbf V - z, z'\rangle_R = \langle \mathbf V - Rz, z'\rangle
\]
for all $z'\in Z_{d}(X;\rho)$. Hence, $\mathbf V - Rz \in B^d(\Gamma;\rho)$.
The uniqueness of $z$ is a consequence of the fact that $B^{d}(\Gamma;\rho)$ is the orthogonal
complement to $Z_d(X;\rho)$ in the standard inner product.

The proof of the last part is given by direct calculation using the self-adjointness
of the operator $\sum_T w_T\bar T$:
\begin{align*}
\langle z,b\rangle &= \frac{1}{r_b} \langle z,b\rangle_R \, , \\
& = \frac{1}{r_b} \langle \frac{1}{\Delta}{\textstyle \sum}_T w_T R^{-1}\mathbf V,b\rangle_R \, ,\\
&= \frac{1}{\Delta}\sum_T \frac{w_T}{r_b} \langle R^{-1}\mathbf V,\bar T(b)\rangle_R \, , \\
&= \frac{1}{\Delta}\sum_T \frac{w_T}{r_b} \langle \mathbf V,\bar T(b)\rangle\, . \qedhere
\end{align*}
\end{proof}

\section{Proof of Theorem \ref{bigthm:twmtt}\label{sec:matrix-tree}}

The proof of Theorem \ref{bigthm:twmtt} is essentially the same
as the proof of \cite[th.~C]{CCK}. We will outline the essential steps.
The first step is to show that 
\begin{equation} \label{eqn:up-to-prefactor}
\det (\partial\partial_R^* \: C_0(\Gamma;\rho) \to C_0(\Gamma;\rho)) = \gamma 
\sum_T w_T
\end{equation}
where $T$ ranges over all $\rho$-spanning trees, and the pre-factor $\gamma$ is to be determined.  This step follows, {\it mutatis mudandis}, by the proof of 
\cite[prop.~4.2]{CCK}. We emphasize that $\gamma$ is independent of $R$.

The second and final step is to compute the prefactor $\gamma$ and show that it
equals 1. We  work perturbatively, following a modified version of
\cite[prop.~5.2]{CCK}.  To this end, let $\beta \in \R_{+}$ be the perturbation
parameter and fix a $\rho$-spanning tree $T$.  For any $W:\Gamma_1 \to \R$,
write $R = e^W$, $R_{\beta} = e^{\beta W}$, and set $\cal L_R = \partial \partial^*_R$. Define $\cal L_R^T
= \partial_T e^{-W}\partial_T^* : C_0(T; \rho) \rightarrow
C_0(T;\rho)$.

A choice of orthogonal projection $C_1(\Gamma;\rho) \rightarrow C_1(T;\rho)$ 
allows us to write
\[
  \cal L_R = \cal L_R^T + \delta \cal L.
\]
A standard expansion of the above operator allows
us to bound the elements of $\delta \cal L$
\[
  |\delta \cal L_{jk} | \leq e^{-\beta \min_{b \in \Gamma_1 \setminus T_1} W_b} B,
\]
where $B$ is independent of $W$ and $\beta$. 
Since $\gamma$ is independent of $R$, we choose 
$W: \Gamma_1 \rightarrow \R$ so that
\[ 
W_{b} > \sum_{\alpha \in T_1} W_{\alpha} -
k \min_{b' \in T_1} W_{b'} \quad \mbox{ for any } b \in \Gamma_1 \setminus T_1,
\]
where $k$ is number of edges of $\Gamma$. 
Our choice of $W$ implies that
in the $\beta \to \infty$ limit, the terms arising from $\cal L_R^T$
dominate those of $\delta \cal L$. Therefore,
\begin{equation} \label{eqn:low-temp}
\lim_{\beta\to \infty} \frac{\det \cal L^T_{R_\beta}}{\det \cal L_{R_\beta}}= 1.
\end{equation}


Substituting $\beta W$ for $W$ in Eqn.~\eqref{eqn:up-to-prefactor},
taking the $\beta\to \infty$ limit, substituting the relation  \eqref{eqn:low-temp} and some minor rewriting, we deduce 
 \[
 \det (\cal L^T_R) = \gamma w_T.
 \]
Note that $\cal L^T_R = \partial_T e^{-W}\partial_T^*$, and 
by definition of $w_T$, we have  $\det e^{-W} = \hat \rho^{-1}_T w_T$.
Consequently, 
\[
\det (\cal L^T_R)  = \hat \rho^{-1}_T w_T \det(\partial_T\partial^*_T).
\]
It follows that
\[
\gamma = \hat \rho^{-1}_T \det(\partial_T\partial^*_T).
\]
Theorem \ref{bigthm:twmtt} is then a consequence of the following.

\begin{lem} For any $\rho$-spanning tree $T$, we have
 \[
 \det (\partial_T\partial^*_T) = \hat\rho_T \, .
 \]
 Hence, $\gamma = 1$.
\end{lem}

\begin{proof}
Clearly both sides of the equation factor as a product of over the connected
components of $T$.
So if $T^\alpha$ is a component of $T$, it will suffice to show 
 \[
 \det (\partial_{T^\alpha}\partial^*_{T^\alpha}) = \hat\rho_{T^\alpha} \, .
\]
This last statement can be proved in a number of ways. For example,
Kenyon \cite{Kenyon}
proves it using an interpretation of the determinant as a summation of cycles
over the symmetric group. We will give
a proof using gauge invariance.

The {\it gauge group} $G$ of $\Gamma$ is the group of functions $\Gamma_0 \to U(1)$
with respect to pointwise multiplication. It is convenient in what
follows to set $g_v = g(v)$ for a vertex $v$.
Then $G$ acts on line bundles according to the rule 
\[
g\cdot \rho(b) = g_{d_0(b)}g_{d_1(b)}^*\rho_b\,.
\]
Set $\rho^g = g\cdot \rho$.
To distinguish between boundary operators, we write
$\partial$ for the boundary operator associated with $\rho$,
and $\partial_g$ for the one associated with $\rho^g$.
Define an action 
\[
G \times C_0(\Gamma;\rho) \to C_0(\Gamma;\rho)
\]
by $g\cdot v = g_v v$, for $v\in \Gamma_0$.

It is then straightforward to check that for $g\in G$ we have
\[
\partial_g\partial_g^* = g \partial\partial^* g^{-1}.
\]
In particular, $\det (\partial_g\partial_g^*) = \det(\partial \partial^*)$.

Write $T^\alpha = A \cup b$, where $A$ is 
tree in the classical sense. We claim that there is a gauge $g\in G$ such that
$\partial_g(e) = 1$ for $e\in A_1$. To find $g$ we need to know that the system
of equations
\begin{equation} \label{eqn:equations}
g_{d_0(e)}g_{d_1(e)}^* \rho_e = 1, \qquad e \in A
\end{equation}
admits a solution. If we fix a vertex $i \in A_0$, we can set $g_i = 1$. Then
for any edge $ij$ of $A$ which connects $i$ to $j$, we set $g_j = \rho_{ij}^s$
with $s = \pm 1$ according as to whether $ij$ points inward towards $i$ or not.
Consider a vertex $k\ne i$ such that $jk$ is an edge of $A$. We set
$g_k = \rho_{jk}^sg_j$, where in this instance $s$ is $\pm 1$ according as
to whether the edge $jk$ points towards $j$ or not. Continuing in this fashion,
we obtain a solution to the system \eqref{eqn:equations}. For any
vertex $i$ not in $A$ we set $g_i = 1$. 
With respect to our choice of $g$, inspection shows that $\rho^g_b$ is the holonomy with respect to $\rho$ 
around the unique circuit of $A \cup b$ which is oriented in the direction of $b$. 
Hence, we can without loss in generality assume that the
original line bundle $\rho$ is trivial on every edge other than
$b$, and we are reduced to proving that 
$\det(\partial\partial^*) =(\rho_b -1)(\rho_b^*-1)$, where
$\partial$ is the boundary operator for $T^\alpha = A\cup b$.

The columns of the matrix associated with $\partial$
represent the edges of $A\cup b$ and 
the rows represent the vertices.
An edge $e$ of $A\cup b$ is said to be {\it loose} if it is attached to a vertex $i$
such that no other edge of $A\cup b$ is attached to $i$. If $e$ is loose, 
then the $e$-th column 
of $\partial$ has exactly two non-zero entries which are $\pm 1$ and these are
of opposite sign. We infer that the determinant of $\partial$ 
remains unchanged when we remove the edge $e$ and the vertex $i$ from $A\cup b$.
Iterating this procedure, we may assume without loss in generality that 
$A\cup b$ has no free edges. This means $A\cup b$ is a circuit. The determinant
of $\partial$ in this case is easy to compute and is given by $\pm(\rho_b -1)$.
Hence the determinant of $\partial^*\partial$ is 
$(\rho_b-1)(\rho_b^*-1)$.
\end{proof}


\begin{thebibliography}{CKS2}
\bibliographystyle{invent}

\bibitem[B]{Bollobas}%
 Bollob\'as, B.: Modern graph theory.
 \newblock Graduate Texts in Mathematics, 184. Springer-Verlag, New York, 1998.


\bibitem[CCK]{CCK}%
Catanzaro, M.J., Chernyak, V.Y., Klein, J.R.:
Kirchhoff's theorems in higher dimensions and Reidemeister torsion.
\newblock {\it arXiv preprint 1206.6783}

\bibitem[CKS1]{CKS1}%
 Chernyak, V.Y., Klein, J.R., Sinitsyn, N.A.:
 Quantization and Fractional Quantization of Currents in Periodically Driven Stochastic Systems I: Average Currents. 
 \newblock {\it J. Chem. Phys.} {\bf 136}, 154107 (2012).

 \bibitem[CKS2]{CKS2}%
 Chernyak, V.Y., Klein, J.R., Sinitsyn, N.A.:
 Quantization and Fractional Quantization of Currents in Periodically Driven Stochastic 
 \newblock {\it J. Chem. Phys.} {\bf 136}, 154108 (2012).




\bibitem[E]{Eckmann}%
Eckmann, B.: Harmonische Funktionen und Randwertaufgaben in einem Komplex.
\newblock{\it Comment. Math. Helv.} {\bf 17} (1945), 240--255.

\bibitem[F]{Forman}%
Forman, R.: Determinants of Laplacians on graphs. 
\newblock{\it Topology}  {\bf 32} (1993), 35�-46



\bibitem[Ke]{Kenyon}%
Kenyon, R.:
Spanning forests and the vector bundle Laplacian
\newblock  {\it Ann. Probab.} {\bf 39} (2011), 1983-�2017



%



\bibitem[NS]{NS} %
Nerode, A., Shank H.: An algebraic proof of Kirchhoff's network theorem,
\newblock {\it Amer. Math. Monthly} {\bf 68}, (1961) 244--247.





\bibitem[R]{Roth}%
 Roth, J.P.:
An application of algebraic topology to numerical analysis: on the existence of a solution to the network problem.
\newblock{\it Proc. Nat. Acad. Sci. U.S.A.} {\bf 41} (1955), 518--521




\end{thebibliography}
\end{document}